\documentclass [11pt]{article}
\usepackage{amssymb,amsmath,comment,amsthm}

\def\N{\mathbb{N}}

\def\F{\mathbb{F}}

\newtheorem{theorem}{Theorem}[section]
\newtheorem{proposition}[theorem]{Proposition}
\newtheorem{corollary}[theorem]{Corollary}
\newtheorem{lemma}[theorem]{Lemma}
\newtheorem{definition}[theorem]{Definition}
\newtheorem{remark}[theorem]{Remark}
\newtheorem{remarks}[theorem]{Remarks}

\newtheorem{examples}[theorem]{Examples}
\newtheorem{coroll-def}[theorem]{Corollary-Definition}

\begin{document}
\title{Admissible family for binary perfect polynomials}
\author{ Luis H. Gallardo  and Olivier Rahavandrainy\\
Univ. Brest, UMR CNRS 6205\\
Laboratoire de Math\'ematiques de Bretagne Atlantique\\
6, Avenue Le Gorgeu, C.S. 93837, 29238 Brest Cedex 3, France.\\
e-mail: Luis.Gallardo@univ-brest.fr \\
Olivier.Rahavandrainy@univ-brest.fr}
\maketitle

\begin{itemize}
\item[a)]
Running head: Odd prime divisors
\item[b)]
\item[c)]
\item[d)]
Corresponding author:
\begin{center} Luis H. Gallardo
\end{center}
\end{itemize}
\newpage~\\

\begin{abstract}
The paper is about an arithmetic problem in $\F_2[x]$. We give \emph{admissible} (necessary) conditions satisfied by a set of odd prime divisors of perfect polynomials over $\F_2$.
This allows us to prove a new characterization of \emph{all} known
perfect polynomials, and to open a way of finding more of them (if they exist).
\end{abstract}

\section{Introduction} \label{intro}
Let $A \in \F_2[x]$ be a nonzero polynomial. We say that $A$ is \emph{even} if it has a linear factor
and it is \emph{odd}, otherwise. We define a
\emph{Mersenne prime} over $\F_2$ as an irreducible polynomial of the form
$1+x^a(x+1)^b$, for some positive integers $a,b$. More generally, we define a \emph{prime} as an irreducible polynomial.
See \cite{Gall-Rahav-mersenn}, for links between Mersenne primes and irreducible binary trinomials.
We denote by $\omega(A)$ (resp. $\sigma(A)$) the number of distinct irreducible factors (resp. the sum of all divisors) of $A$ over $\F_2$
($\sigma$ is a multiplicative
function). $A$ \emph{splits} if $A$ is even and $\omega(A) \leq 2$. We call $A$ \emph{perfect} if $\sigma(A) = A$. Finally, a perfect polynomial is \emph{indecomposable} if it does not factor in two coprime nonconstant perfect ones.

We also denote by:\\
- $\rm{rad}(A)$, the \emph{radical} of $A$: the product of all the distinct prime divisors of $A$ in $\F_2[x]$,\\
- $\N$ ($\N^*$), the set of (positive) natural numbers,\\
- $A'$, the formal derivative of $A \in \F_2[x]$ relative to $x$.

Given $k \in \N^*$ and $A, P \in \F_2[x]$ with $P$ irreducible, we write:
$$\text{$P^k \| A$ if $P^k \mid A$ but $P^{k+1} \nmid A.$}$$
For $Q \in \F_2[x]$ odd, we put $Q^{\langle{a,b,c}\rangle} := 1+x^a(x+1)^bQ^c$, $\overline{Q}$ the polynomial obtained from $Q$, by substituting $x$ by $x+1$
and $\displaystyle{Q^*(x) := x^{\deg(Q)} \cdot Q(\frac{1}{x})}$ (the reciprocal of $Q$). We remark that $\overline{Q^{\langle{a,b,c}\rangle}} = \overline{Q}\ ^{\langle{b,a,c}\rangle}.$

Polynomials below are important in our work. The $M_j$'s and the $S_k$'s are all irreducible (see Lemma \ref{allirreduc}).\\
$\begin{array}{l}
M_1=1+x+x^2,\ M_2=1+x+x^3,\ M_3=\overline{M_2}=1+x^2+x^3,\\
M_4=1+x+x^2+x^3+x^4, M_5=\overline{M_4} = 1+x^3+x^4,\
M_6=1+x^3+x^5,\\
M_7=1+x^3+x^7,\ M_8=1+x^6+x^7,\
M_9=\overline{M_6},\  M_{10}=\overline{M_7},\  M_{11}=\overline{M_8},\\
M_{12} = x^9+x+1, M_{13}=\overline{M_{12}}=x^9+x^8+1,
\end{array}$\\
$\begin{array}{l}
T_1 =x^2(x+1)M_1, \ T_2=\overline{T_1},\
T_3 = x^4(x+1)^3M_4,\ T_4 =\overline{T_3},\\
T_5 = x^4(x+1)^4M_4\overline{M_4} = \overline{T_5},\
T_6 = x^6(x+1)^3M_2\overline{M_2}, \ T_7= \overline{T_6},\\
T_8 = x^4(x+1)^6M_2\overline{M_2} M_4,\ T_9 = \overline{T_8},\\
T_{10} = x^2(x+1){M_1}^2 (1+x+x^4),\ T_{11} = \overline{T_{10}},
\end{array}$\\
$\begin{array}{l}
S_1 ={M_1}^{\langle{1,1,1}\rangle} = \overline{S_1},\ S_2 = {M_1}^{\langle{2,2,1}\rangle},\
S_3  = {M_1}^{\langle{1,3,4}\rangle}, \ S_4 = {M_1}^{\langle{3,1,1}\rangle}, \\
S_5  = {M_1}^{\langle{1,3,1}\rangle},\
S_{6} = {M_1}^{\langle{3,1,4}\rangle},\ S_{7} = {M_1}^{\langle{1,1,3}\rangle}, \
S_{8} = {M_1}^{\langle{3,3,1}\rangle},\\
S_9 = {M_1}^{\langle{1,1,5}\rangle}, \
S_{10} = {M_1}^{\langle{4,1,1}\rangle},\
S_{11} = {M_1}^{\langle{1,2,1}\rangle}, \ S_{12} = {M_1}^{\langle{2,1,2}\rangle}, \\
S_{13} = {M_1}^{\langle{1,4,1}\rangle},\
S_{14} = {M_1}^{\langle{2,1,1}\rangle}, \ S_{15} = {M_1}^{\langle{1,2,2}\rangle}.
\end{array}$\\
We set ${\cal{F}}_1 := \{M_1, \ldots, M_{13}\}$, ${\cal{F}}_2 :=\{S_1,\ldots,S_{15}\}$ and ${\cal{F}} := {\cal{F}}_1 \cup {\cal{F}}_2$.\\

The following facts are well-known \cite{Canaday}. Besides \emph{trivial} perfects (of the form  $x^{2^n-1}(x+1)^{2^n-1}$, with $n \in \N^*$), there are only $11$ known perfects, all of them are even, namely $T_1,\ldots, T_{11}$. There is no other perfect polynomial $A$ with $\omega(A)<5$ (see \cite{Gall-Rahav4,Gall-Rahav5,Gall-Rahav7}). Recently, Cengiz et al. \cite{Po} proved by extensive computations that there is no other perfect polynomial $A$ with $\deg(A) \leq 200$.

Odd prime factors of the $T_j$'s are all Mersenne primes, except: $S_1 =1+x+x^4 = 1+x(x+1)M_{1}$. More precisely, $T_1, \ldots, T_9$ are the unique perfects of the form
$\displaystyle{x^a(x+1)^b \prod_j {P_j}^{h_j}}$, with all the $P_j$'s Mersenne primes and $a,b,h_j \in \N$ (\cite[Theorem 1.1]{Gall-Rahav16}). The last two: $T_{10}$ and $T_{11}$ are the unique of the form $x^a(x+1)^b M^{2h} \sigma(M^{2h})$, with $M$ a Mersenne prime and $a,b, h \in \N^*$ (\cite[Theorem 1.4]{Gall-Rahav13}).

We would like to extend the set of such odd primes (\emph{admissible family}) in order to discover new perfect polynomials. In this paper, we consider the family ${\cal{F}}$ defined above. We recall in Section \ref{admissible},
how and why we choose its members: $M_1,\ldots, M_{13}$ and $S_1, \ldots, S_{15}$. For more details, see \cite{Gall-Rahav15}.

Canaday \cite[Theorem 16, Theorem 20]{Canaday}
stated that some even (resp. odd) perfect polynomial $A$
with special factorization is uniquely determined by the exponents of $x$ and of $x+1$
(resp. by any odd prime divisor of $A$).
Our goal is to prove that if the radical of an even non-splitting perfect polynomial factors in $\{x,x+1\} \cup {\cal{F}}$, then we exactly get those eleven known (no more ones). Perhaps, by choosing a bigger admissible family, one would obtain new perfect polynomials...

\begin{theorem} \label{mainresult1}
Let $A$ be an even non-splitting binary polynomial,
with all odd prime divisors in  ${\cal{F}}$.
Then, $A$ is indecomposable perfect
if and only if $A, \overline{A} \in \{T_1, \ldots, T_{11}\}$.
\end{theorem}

The proof of this theorem shows a ``kind of algorithm'' to give (at most) even perfect polynomials with a given admissible family.

Our method requires some simple computer calculations. So, we believe that it should be able to find some new perfect polynomials $A$ (with $\omega(A)$ or $\deg(A)$ moderate large), if they exist.

By the same method, in \cite{Rahav}, we can characterize all the known even non-splitting unitary perfect polynomials over $\F_2$ (listed in \cite{BeardU2}) and we discover many new ones.

\begin{remark} \label{nosotrivial}
\emph{For a given admissible family ${\cal{G}}$, a binary polynomial $A$ such that $\rm{rad}(A)$ is a product of members of ${\cal{G}}$, may have a potentially arbitrary factor of the form $Q^{m}$, with $m \in \N^{*}$.
In other words, Theorem \ref{mainresult1} requires some work to be proved, although we assume that $\omega(A) \leq 30$ (instead to be an arbitrary positive integer).}
\end{remark}

\section{Useful facts} \label{usefulfacts}

\subsection{Admissible family} \label{admissible}

We get Definition \ref{defadmissible} and Corollary-Definition \ref{oddadmissible}, inspired by Lemmas \ref{factorx2h} and \ref{bar-star-stable}.
\begin{lemma} \label{factorx2h}
Let $h \in \N^*$. Then, for any prime factor $P$ of $\sigma(x^{2h})$, $P^*$ $($resp. $\overline{P})$ also divides $\sigma(x^{2h})$ $($resp. $\sigma((x+1)^{2h}))$.
\end{lemma}
\begin{proof}
We remark that $(\sigma(x^{2h}))^* = \sigma(x^{2h})$. So, for any irreducible factor $U$ of $\displaystyle{\sigma(x^{2h})}$, $U^*$ also divides $\sigma(x^{2h})$. Our result follows.
\end{proof}
\begin{lemma} \label{bar-star-stable}
Let $B$ be an even non splitting perfect polynomial over $\F_2$ and $Q$ an odd prime divisor of $B$. Then:\\
i) there exists $h \in \N^*$ such that $x^{2h}$ or $(x+1)^{2h}$ divides $B$,\\
ii) $1+Q$ divides $B$ or $\sigma(Q^{2h})$ divides $B$, for some $h \in \N^*$.
\end{lemma}

\begin{proof}
i): $B$ does not split, so the exponent of $x$ (resp. of $x+1$) in $B$ is of the form $2^{t_1}s_1-1$ (resp. $2^{t_2}s_2-1$), where $s_1, s_2$ are odd, $s_1 \geq 3$ or $s_2 \geq 3$, and $\sigma(x^{s_1-1})$ or $\sigma((x+1)^{s_2-1})$ divides $\sigma(B)=B$. Take then: $2h=s_1-1$ or $s_2-1$.\\
ii): The exponent of $Q$ in $B$ is of the form $2^ts-1$, with $s$ odd and $t \geq 1$. If $s=1$, then $1+Q$ divides $(1+Q)^{2^t-1} = \sigma(Q^{2^t-1})$ which in turn, divides $\sigma(B)=B$. If $s\geq 3$, then $\sigma(Q^{s-1})$ divides $\sigma(Q^{2^ts-1})$ and $B$. Thus, take $h=\displaystyle{\frac{s-1}{2}}$.
\end{proof}

\begin{definition} \label{defadmissible}
\emph{A family ${\cal{G}}$ of odd irreducible polynomials is} admissible \emph{if it satisfies at least i), ii) or iii):\\
i) For any $T \in {\cal{G}}$,
$T^* \in {\cal{G}}$ or $\overline{T} \in {\cal{G}}$.\\
ii) There exists $h \in \N^*$ such that $\sigma(x^{2h})$ or $\sigma((x+1)^{2h})$  factors in ${\cal{G}}$.\\
iii) For any $T \in {\cal{G}}$, $1+T$ or $\sigma(T^{2h})$ factors in  ${\cal{G}} \cup \{x,x+1\}$, for some $h \in \N^*$.}
\end{definition}

\begin{coroll-def} \label{oddadmissible}~\\
The set of odd prime divisor(s) of any even non-splitting perfect polynomial $A$ is admissible, called {\tt{admissible family for $A$}}.
\end{coroll-def}
By direct computations, we give
\begin{lemma} \label{allirreduc}
The polynomials $M_j$'s and $S_k$'s defined at the beginning of Section \ref{intro} are all irreducible. Moreover, each $M_j$ is a Mersenne one.
\end{lemma}
\begin{remarks}~\\
\emph{i) An admissible family is not necessarily stable both under $Q \mapsto \overline{Q}$ and $Q \mapsto Q^*$.
For example, ${\cal{G}} = \{M_1, \ldots, M_5\}$ is admissible giving the first nine perfect polynomials $T_1, \ldots, T_9$. However, ${M_5}^* = S_1 \not\in {\cal{G}}$.}\\
\emph{ii) The converse of Corollary-Definition \ref{oddadmissible} is false: $\{M_2\}$ is admissible ({\scriptsize{iii) satisfied}}), but there exists no perfect polynomial of the form $x^a(x+1)^b{M_2}^c$.}
\end{remarks}

\begin{examples} \label{famillexemple}
$$\begin{array}{|l|l|}
\hline
\text{Admissible family}&\text{Associated even perfect(s)}\\
\hline
\ \emptyset & \text{Trivial ones}\\
\{M_1\}&T_1, T_2\\
\{M_2\}&\text{No one}\\
\{M_4\}&T_3\\
\{M_5\}&T_4\\
\{M_2,M_3\}&T_6, T_7\\
\{M_4,M_5\}&T_5\\
\{M_2,M_3,M_4\}&T_8\\
\{M_2,M_3,M_5\}&T_9\\
\{M_1,\ldots,M_5\}&T_1,\ldots,T_9\\
\{M_1,S_1\}&T_{10}, T_{11}\\
\hline
\end{array}$$
\emph{$\{M_1\}$, $\{M_4\}$ satisfy i), ii) and iii).
$\{M_5\}$ satisfies ii) and iii), but not i).}
\end{examples}

\begin{proposition} \label{Mersenneadmissible}
The set of all Mersenne primes is admissible and admits $T_1, \ldots, T_9$ as associated even perfects.
\end{proposition}
\begin{proof}
The part i) of Definition \ref{defadmissible} is satisfied: if $M$ is a Mersenne prime, then $\overline{M}$ is also a Mersenne prime. See then \cite[Theorem 1.1]{Gall-Rahav16}.
\end{proof}

\subsection{The family ${\cal{F}}$} \label{thefamily}
We sketch how we choose the family ${\cal{F}} = {\cal{F}}_1 \cup {\cal{F}}_2$ (see \cite{Gall-Rahav15}).
We begin with the \emph{reciprocity stability} in order to get the first members:
$$M_1, \ldots, M_4, M_{12}, M_{13}, S_1, S_{2}, S_{3}, \ldots$$
After that, for $S  \in \{x, x+1\}$, for $S$ Mersenne prime or for $S$ of the form ${M_1}^{\langle{a,b,c}\rangle}$, we search all prime divisors of some $\sigma(S^{2h})$, $h \in \N^*$. By the way, we are able to find all possible exponents $m$, with $P^m \mid \mid \sigma(A)$. This is the core of the method,
since we have at this step, a finite number of possibilities to try with the computer.

In this section, we suppose that $Q=M_1$ and that $Q^{\langle{a,b,c}\rangle}$ is irreducible. So,
$(Q^{\langle{a,b,c}\rangle})^* = x^{a+b+2c}+(x+1)^b(x^2+x+1)^c$, with $\gcd(a,b,c) =1$.\\
Since $Q^{\langle{a,b,c}\rangle} \in {\cal{F}}$ implies that $\overline{Q^{\langle{a,b,c}\rangle}} \in {\cal{F}}$, we also require that $(Q^{\langle{a,b,c}\rangle})^* \in {\cal{F}}$, in order to get a bigger admissible family. Nevertheless, we are limited in our choice because of the difficulty to prove polynomial irreducibility.\\
So, we consider three cases:
$(Q^{\langle{a,b,c}\rangle})^*$ is Mersenne, $(Q^{\langle{a,b,c}\rangle})^* = Q^{\langle{a,b,c}\rangle}$ and $(Q^{\langle{a,b,c}\rangle})^*  = Q^{\langle{d,e,f}\rangle} \not= Q^{\langle{a,b,c}\rangle}$.\\
The first Mersenne prime members of ${\cal{F}}$ are obtained from Lemma \ref{Merstar=Mers}, whereas the other members, from Section 3-4 in \cite{Gall-Rahav15}. More precisely, in \cite{Gall-Rahav15}: \\
- Section 3-4-1 gives $S_1, S_{10}, S_{14}, S_{15}$, with $(S_1)^*=M_5$, $(S_{10})^* = M_7$, $(S_{14})^* = M_6$, and $(S_{15})^* = M_8$,\\
- Proposition 3-15 gives $S_3= (S_3)^*$ and $S_4 = (S_4)^*$,\\
- from Section 3-4-3, we get $S_2, S_{5}, S_{6}, S_9$, with $(S_2)^*=S_5$ and $(S_{6})^* = S_9$,\\
- we take $S_7$ and $S_8$ because $\sigma(M_1^4) = S_8$ and $\sigma(S_2^2) = S_1S_7$,\\
- we finally add $S_{11} = \overline{S_{14}}$, $S_{13} = \overline{S_{10}}$ and $S_{12} = \overline{S_{15}}$.

\begin{lemma} {\rm{(\cite[p. 728-729]{Canaday})}} \label{Merstar=Mers}~\\
Let $M$ be a Mersenne prime such that $M^*$ is also Mersenne. Then\\
i) $M \in \{M_1, M_4\}$ if $M=M^*$.\\
ii)  $M \in \{M_2, M_3,M_{12}, M_{13}\}$ if $M \not= M^*$.
\end{lemma}

\begin{corollary} \label{Fisadmissible}
The family ${\cal{F}}$ is admissible.
\end{corollary}

\begin{proof}
The condition i) in Definition \ref{defadmissible} is obviously satisfied.
\end{proof}

\begin{remarks} \label{remarks3}~\\
\emph{i) By direct computations, the sum $\displaystyle{\sum_{D \in {\cal{F}}} \deg(D)}$ equals $184$.}\\
\emph{ii) For any $T \in {\cal{F}}$, one has: $\overline{T} \in {\cal{F}}$.\\
iii)
$\text{$T^* \not\in {\cal{F}}$ if $T \in \{M_9,M_{10}, M_{11}, S_7, S_8, S_{11},S_{12},S_{13}\}$}.$\\
iv) ${\cal{F}}$ contains all the families described in Examples \ref{famillexemple} and some primes of the form ${M_1}^{\langle{a,b,c}\rangle}$, like $S_1$.}
\end{remarks}
We take:
\begin{equation} \label{formofA}
\displaystyle{A = x^a(x+1)^b \prod_{i=1}^{13} {M_i}^{c_i} \cdot \prod_{j=1}^{15} {S_j}^{d_j}=x^a(x+1)^b \ A_1},
\end{equation}
where $a,b, c_i, d_j \in \N, \ a,b \geq 1$ and $A_1 \not= 1$ (so that $\omega(A) \leq  30$).\\
We also put:
\begin{equation} \label{lesexposants}
a =2^nu-1,\ b=2^mv-1,\ c_i = 2^{n_i}u_i-1,\ d_j = 2^{m_j}v_j-1, i \leq 13, j \leq 15,
\end{equation}
for some odd integers $u,v, u_i, v_j$, and for some $n,m, n_i, m_j \in \N$.

\subsection{Prime divisors of $\sigma(A)$ and their exponents} \label{lesdiviseursdesigmA}

In order to compare $A$ and $\sigma(A)$, we give all prime divisors of $\sigma(A)$ with their exponents. With the same notations as in (\ref{formofA}) and in (\ref{lesexposants}), we may write:
\begin{equation} \label{sigmaAdetails}
\begin{array}{l}
\displaystyle{\sigma(A) = \sigma(x^a) \sigma((x+1)^b)) \prod_{i=1}^{13} \sigma({M_i}^{c_i}) \prod_{j=1}^{15} \sigma({S_j}^{d_j})},\\
\sigma(x^a) = (x+1)^{2^n-1} \cdot [\sigma(x^{u-1})]^{2^n},\
\sigma((x+1)^b) = x^{2^m-1} \cdot [\sigma((x+1)^{u-1})]^{2^m},\\
\sigma({M_i}^{c_i}) = (1+M_i)^{2^{n_i}-1} \cdot [\sigma({M_i}^{u_i-1})]^{2^{n_i}},\\
\sigma({S_j}^{d_j}) = (1+S_j)^{2^{m_j}-1} \cdot [\sigma({S_j}^{v_j-1})]^{2^{m_j}}.
\end{array}
\end{equation}
We must find all $h \in \N^*$ such that $\sigma(S^{2h})$ factors in ${\cal{F}}$, for $S \in \{x,x+1\} \cup {\cal{F}}$. Assuming we obtained these values of $h$, we put:
\begin{equation}
\label{expressionsigmA}
\displaystyle{\sigma(A) = x^{\alpha} (x+1)^{\beta} \prod_{i=1}^{13} {M_i}^{\gamma_i} \prod_{j=1}^{15} {S_j}^{\delta_j}},
\text{ where $\alpha, \beta, \gamma_i,\ \delta_j \in \N$}.
\end{equation}

\begin{lemma}
\label{divisorofTT'}
For any $S \in {\cal{F}}_2$ and $h \in \N^*$, $M_1$ does not divide $\sigma({S}^{2h})$.
\end{lemma}

\begin{proof}
Keep in mind that any element of ${\cal{F}}_2$ is irreducible. Put $S =1+x^c (x+1)^d {M_1}^e$. If $\alpha$ is a root of $M_1$, then $1 =1+0 = 1+\alpha^c (\alpha+1)^d (M_1(\alpha))^e= S(\alpha)$ and so $(\sigma({S}^{2h}))(\alpha) = 1+S(\alpha) + \cdots + (S(\alpha))^{2h} =  1 \not= 0$.
\end{proof}

\begin{lemma}
\label{S2h-squarefree}
For any $h \in \N^*$ and for any $S \in \{x,x+1\} \cup {\cal{F}}$, $\sigma(S^{2h})$ is odd and square-free.
\end{lemma}

\begin{proof}
Obviously, $\sigma(S^{2h})$ is odd. Moreover, $\sigma(S^{2h})$ is square-free if $S \in \{x,x+1\} \cup {\cal{F}}_1$
(\cite[Lemma 2.6]{Gall-Rahav13}). Now, consider $S = {M_1}^{\langle{a,b,c}\rangle}=1+x^a (x+1)^b {M_1}^c \in {\cal{F}}_2$. Put $T = \sigma({S}^{2h}) = (1+S)(1+S +\cdots + {S}^{h-1})^2 + {S}^{2h}$. One has $T' = S' \cdot (1+S +\cdots + {S}^{h-1})^2$. We claim that $\gcd(T,T') = 1$. Let $D$ be a common prime divisor of $T$ and $T'$. If $D$ divides $1+S +\cdots + {S}^{h-1}$, then $D$ divides ${S}^{2h}$ and hence $D=1$.
If $D$ divides $S'$, then by direct computations, $D \in \{1, M_1\}$ because $D$ is odd. Thus $D=1$, by Lemma \ref{divisorofTT'}.
\end{proof}

\begin{lemma}
\label{divsigmx2h}
If $\sigma(x^{2h})$ and $\sigma((x+1)^{2h})$ factor in ${\cal{F}}$, then $2h \in \{2,4,6,8,12,14\}$.
In this case,
$$\begin{array}{l}
\sigma(x^{2}) = \sigma((x+1)^2) = M_1,\ \sigma(x^{4}) = M_4, \ \sigma((x+1)^{4}) = M_5, \\
\sigma(x^{6}) = \sigma((x+1)^{6}) =  M_2 M_3,\
\sigma(x^{8}) = M_1S_4, \ \sigma((x+1)^{8}) = M_1S_5, \\
\sigma(x^{12}) = S_3,\ \sigma((x+1)^{12})=S_6,\
\sigma(x^{14}) = \sigma((x+1)^{14}) = M_1M_4M_5S_1.
\end{array}$$
\end{lemma}

\begin{proof}
We remark that $\sigma((x+1)^{2h}) = \overline{\sigma(x^{2h})}$. So, it suffices to consider $X_h:=\sigma(x^{2h})$. One has: $\displaystyle{X_h = \prod_{P \in {\cal{F}}} P^{c_P}}$, where $c_P \in \{0,1\}$, because $X_h$ is square-free. Moreover, $2h = \deg(x^{2h}) \leq 184$, by Remarks \ref{remarks3}-i). Direct (Maple) computations (which are done, for $h \leq 92$) prove the result.
\end{proof}

\begin{lemma}
\label{alldivsigmMers2h}
Let $M \in {\cal{F}}_1$ be such that $\sigma(M^{2h})$ factors in ${\cal{F}}$. Then, $(M=M_1$ and $2h \in \{2,4,6,14\})$ or $(M \in \{M_2, M_3\}$ and $2h=2)$. We get: \\
$\sigma({M_2}^2) = M_1M_5,\ \sigma({M_3}^2) = M_1M_4$,
$\sigma({M_1}^2) = S_1,\ \sigma({M_1}^4) = S_{8}$,\\
$\sigma({M_1}^6) = M_2M_3S_2,\ \sigma({M_1}^{14}) = M_4M_5S_1S_{7}S_{8}.$
\end{lemma}

\begin{proof} As above, we may write
$\displaystyle{\sigma(M^{2h}) = \prod_{P \in {\cal{F}}} P^{c_P}}$, with $c_P \in \{0,1\}$ and $4h \leq 2h\deg(M) \leq 184$. So, $h \leq 46$.
Direct computations (which took about 30 min.) prove our result.
\end{proof}

\begin{lemma}
\label{divsigmS2h}
Let $S \in {\cal{F}}_2$ be such that $\sigma({S}^{2h})$ factors in ${\cal{F}}$, then $2h=2$, $S \in \{S_1,S_2\}$,
$\sigma({S_1}^{2}) = M_4M_5$ and $\sigma({S_2}^{2}) = S_1S_{7}$.
\end{lemma}

\begin{proof}
Analogous proof: here, $8h \leq 2h \deg(S) \leq 184$. So, $h \leq 23$ (computations took 125 s).
\end{proof}

Lemmas \ref{divsigmx2h}, \ref{alldivsigmMers2h} and \ref{divsigmS2h} imply:

\begin{corollary}
\label{theoddivisors}
i) If $M_i$ and $S_j$ divide $\sigma(A)$, then $i \leq 5$ and $j \leq 8$.\\
ii)  For any $j \in \{2,\ldots,6\}$, ${S_j}^2$ does not divide $\sigma(A)$.
\end{corollary}

\begin{proof}
i): For any $i \geq 6$ and $j \geq 9$, neither $M_i$ nor $S_j$ divides $\sigma(A)$.\\
ii): $S_2, S_3, S_4, S_5, S_6$ respectively divide only $\sigma({M_1}^6), \sigma(x^{12}), \sigma(x^8), \sigma((x+1)^8)$ and $\sigma((x+1)^{12})$. So, for any $j \in \{2,\ldots,6\}$, ${S_j}^2$ does not divide $\sigma(A)$.
\end{proof}
For $w \in \N^*$, $\chi_{w}$ denotes the indicator function of the singleton $\{w\}$:
$$\text{$\chi_{w}(w) = 1, \chi_{w}(t) = 0$ if $ t\not= w$.}$$
According to notations in (\ref{lesexposants}), put:
$$\begin{array}{l}
\text{$M_i = 1+x^{a_i} (x+1)^{b_i} \in {\cal{F}}_1$, $S_j = 1+x^{\alpha_j} (x+1)^{\beta_j} {M_1}^{\nu_j} \in {\cal{F}}_2$},\\
\text{$\xi_1 = \chi_{3}(u) + \chi_{9}(u) + \chi_{15}(u), \ \xi_2 = \chi_{3}(v) + \chi_{9}(v) + \chi_{15}(v)$},\\
\text{$\xi_3 = \chi_5(u) + \chi_{15}(u), \ \xi_4 = \chi_5(v) + \chi_{15}(v)$}.
\end{array}
$$
We obtain from (\ref{expressionsigmA}) and equalities in (\ref{sigmaAdetails}):
\begin{lemma}
\label{exponentsofsigmA}
The integers $\alpha, \beta, {\gamma_i}$'s and ${\delta_j}$'s satisfy:
$$\begin{array}{l}
\displaystyle{\alpha = 2^m-1 + \sum_{i=1}^5 (2^{n_i}-1)a_i + \sum_{j=1}^8 (2^{m_j}-1)\alpha_j,}\\
\displaystyle{\beta = 2^n-1 + \sum_{i=1}^5 (2^{n_i}-1)b_i + \sum_{j=1}^8 (2^{m_j}-1)\beta_j,}\\
\displaystyle{\gamma_1 = \sum_{j=1}^8 (2^{m_j}-1)\nu_j + \xi_1 \cdot 2^n + \xi_2 \cdot 2^m + \chi_{3}(u_2) \cdot 2^{n_2}
+\chi_{3}(u_3) \cdot 2^{n_3}},\\
\displaystyle{\gamma_2 = \gamma_3=\chi_7(u) \cdot 2^n + \chi_7(v) \cdot 2^m + \chi_7(u_1) \cdot 2^{n_1},}\\
\displaystyle{\gamma_4 = \xi_3 \cdot 2^n + \chi_{15}(v) \cdot 2^m + \chi_{15}(u_1) \cdot 2^{n_1}+
\chi_{3}(u_3) \cdot 2^{n_3} + \chi_{3}(v_1) \cdot 2^{m_1}},\\
\displaystyle{\gamma_5 = \chi_{15}(u) \cdot 2^n + \xi_4 \cdot 2^m + \chi_{15}(u_1) \cdot 2^{n_1}+
\chi_{3}(u_2) \cdot 2^{n_2} + \chi_{3}(v_1) \cdot 2^{m_1}},\\
\displaystyle{\delta_1 = \chi_{15}(u) \cdot 2^n + \chi_{15}(v) \cdot 2^m + (\chi_{3}(u_1) + \chi_{15}(u_1)) \cdot 2^{n_1}},\\
\delta_2= \chi_{7}(u_1) \cdot 2^{n_1},\ \delta_3= \chi_{13}(u) \cdot 2^{n},\ \delta_4= \chi_{9}(u) \cdot 2^{n},\ \delta_5= \chi_{9}(v) \cdot 2^{m},\\
\delta_6= \chi_{13}(v) \cdot 2^{m},\ \delta_7 = (\chi_{5}(u_1) + \chi_{15}(u_1)) \cdot 2^{n_1}, \ \delta_8 = \chi_{15}(u_1) \cdot 2^{n_1}.
\end{array}$$
\end{lemma}

\subsection{More necessary conditions for $A$ to be perfect} \label{sigmAandA}

We suppose that $A$ is perfect ($A=\sigma(A)$).  We give necessary conditions on the exponent $m$ of each prime divisor $P$ of $A$ (i.e., for $m$ satisfying: $P^m \mid \mid A$).
Those conditions are very useful for computing $A$.
We keep the notations in (\ref{lesexposants}), (\ref{sigmaAdetails}) and (\ref{expressionsigmA}).

\begin{lemma} \label{corol1A=sigmA}
If $A$ is perfect, then:\\
i) $a = 2^n u - 1, b=2^mv-1$ where $n,m \in \N$, $u, v \in \{1,3,5,7,9,13,15\}$ and
$u\geq 3$ or $v \geq 3$.\\
ii)  $c_i = 0$ and $d_j = 0$, for any $i \geq 6$ and $j \geq 9$.\\
iii) $c_1 =2^{n_1} u_1 -1$ where $u_1 \in \{1,3,5,7,15\}$.\\
iv)  $c_i =2^{n_i} u_i -1$, with $u_i \in \{1,3\}$ if $i \in \{2,3\}$, $u_i = 1$ if $i \in \{4,5\}$.\\
v) $d_j = 2^{m_j} v_j -1$ where $v_1 \in \{1,3\}$, $v_j =1$ if $j \in \{2,\ldots,8\}$.
\end{lemma}

\begin{proof}
i): $(x+1)^{2^n-1} \sigma(x^{u-1}) = \sigma(x^a)$ divides $\sigma(A)=A$. So, $\sigma(x^{u-1})$ divides $A$ and $u-1 \in \{0,2,4,6,8,12,14\}$, by Lemma \ref{divsigmx2h}.\\
Analogously, one has: $v-1 \in \{0,2,4,6,8,12,14\}$.\\
If $u=v=1$, then $x^a(x+1)^b$ is perfect. Thus, $A_1$ is odd and perfect, with $A_1 \not=1$. It contradicts the fact that $A$ is indecomposable.\\
Similar arguments from Lemmas \ref{alldivsigmMers2h} and \ref{divsigmS2h} give: ii), iii), iv) and the first part of v).\\
Finally, $S_2$ divides $\sigma({M_1}^{6})$, $S_7$ divides both $\sigma({M_1}^{14})$ and $\sigma({S_2}^2)$. But, by Corollary \ref{theoddivisors}, ${S_2}^2$ does not divide $\sigma(A)=A$. Hence, $v_2 = v_7 = 1$.
\end{proof}

\begin{lemma} \label{aboutn2345m1}
One has: $n_2, n_3, m_1 \leq 3$ and $n_4, n_5 \leq 5$.
\end{lemma}

\begin{proof}
We begin with the condition $2^{m_1}v_1-1 = d_1 = \delta_1 = \varepsilon_1 \cdot 2^n +\varepsilon_2 \cdot 2^m + \varepsilon_3 \cdot 2^{n_1}$, where $\varepsilon_k \in \{0,1\}$.
If $m_1 \geq 1$, then $d_1$ is odd. So, $d_1 =1$ or it is of form $2^{h_1} +1$ or $2^{h_1}+ 2^{h_2} +1$, with $h_1, h_2 \geq 1$. Since $v_1 \in \{1,3\}$, we get $m_1 \leq 3$. In the same manner, $n_2, n_3 \leq 3$.\\
Now, consider $2^{n_4}-1 = c_4 = \gamma_4 = \varepsilon_1 \cdot 2^n +\varepsilon_2 \cdot 2^m + \varepsilon_3 \cdot 2^{n_1}+ \varepsilon_4 \cdot 2^{n_3}+ \varepsilon_5 \cdot 2^{m_1}$, where $\varepsilon_k \in \{0,1\}$ and $m_1, n_3 \leq 3$. If $n_4 \geq 1$, then $c_4$ is odd. So,
$c_4 \in K_1 \cup K_2$, where $K_1 =\{1,3,5,2^{h_1}+1, 2^{h_1}+3, 2^{h_1}+ 2^{h_2} + 1, 2^{h_1}+ 2^{h_2} + 3: h_1, h_2 \geq 1\}$ and
$K_2 =\{2^{h_1}+ 2^{h_2} +2^{h_3} + \ell: \ell \in \{1,3,5,9\}, h_1, h_2, h_3 \geq 1\}$. Maple computations give:
$n_4 \leq 5$. We also have: $n_5 \leq 5$.
\end{proof}

The proof of the following lemma (sketched in \cite{Canaday}) is given in \cite{Gall-Rahav5}.

\begin{lemma} \label{omegaleq4}~\\
If $B$ is an even non splitting perfect polynomial over $\F_2$, with $\omega(B)\leq 4$, then
$B \in \{T_1, T_2, T_3,T_4,T_5, T_6,T_7, T_{10}, T_{11}\}$.
\end{lemma}

\begin{corollary} \label{aboutnmn1}
One has: $n,m ,n_1 \leq 4$.
\end{corollary}

\begin{proof}
We know that $u,v \in \{1,3,5,7,9,13,15\}, u_1 \in \{1,3,5,7,15\}$ and $u_2, u_3, v_1 \in \{1,3\}$, with $u \geq 3$ or $v \geq 3$,
$n_2, n_3, m_1 \leq 3$.\\
- If $u=7$ then from the expression of $\gamma_2 = c_2 = 2^{n_2} u_2 -1$, we get $2^n \leq \gamma_2 =2^{n_2} u_2 -1 \leq 23$. So, $n \leq 4$.\\
- If $u \in \{9,13,15\}$, then  $n=0$ (from the expressions of $\delta_4$, $\delta_3$ and $\gamma_5$).\\
- Analogously, if $v \in \{7,9,13,15\}$, then $m \leq 4$.\\
- If $u_1 = 3$, then $2^{n_1} \leq \delta_1 = d_1 = 2^{m_1} v_1 -1 \leq 23$. So, $n_1 \leq 4$.\\
- If $u_1 \in \{5,7,15\}$, then  $n_1=0$ (from the expressions of $\delta_7$, $\delta_2$ and $\gamma_8$).\\
- It remains the case where $u,v \in \{1,3,5\}$ (with $u \geq 3$ or $v \geq 3$) and $u_1 = 1$. We immediately have: $d_j = \delta_j = 0$ for any $j\geq 1$ and $c_2 =c_3 = \gamma_2 = \gamma_3 = 0$. Thus $A = x^a(x+1)^b {M_1}^{c_1} {M_4}^{c_4} {M_5}^{c_5}$. We may suppose that $u \in \{3,5\}$, and we apply Lemma \ref{omegaleq4}:\\
$\bullet$ If $u= 3$, then $c_4 = \gamma_4 = 0$. So, $\omega(A) \leq 4$, $c_5 =0$ and $A =T_1$, $n =0$, $m=1$.\\
$\bullet$ If $u=5$, then $c_1 = \delta_1 = 0$, $A = x^a(x+1)^b {M_4}^{c_4} {M_5}^{c_5}$. Hence, $A = T_5$ and $n=m=0$.
\end{proof}

\begin{corollary} \label{finalconditions}
If $A$ is perfect, then $u_4 = u_5=1$, $d_j \in \{0,1\}$ for $j \geq 2$ and
$u,v \in \{1,3,5,7,9,13,15\}, u_1 \in \{1,3,5,7,15\}, u_2, u_3, v_1 \in \{1,3\}$,\\
$n,m ,n_1 \leq 4, \  n_2, n_3, m_1 \leq 3, \ n_4, n_5 \leq 5.$
\end{corollary}

\section{Proof of Theorem \ref{mainresult1}}
The conditions are sufficient. Thus, we  prove  that they are necessary. We may write (see notations in (\ref{formofA}) and in (\ref{lesexposants})):
$$\displaystyle{A = x^a(x+1)^b \prod_{i=1}^{13} {M_i}^{c_i} \cdot \prod_{j=1}^{15} {S_j}^{d_j}=x^a(x+1)^b \ A_1},$$
where $\text{$a,b, c_i, d_j \in \N, \ a,b \geq 1$}$,
$a =2^nu-1,\ b=2^mv-1,\ c_i = 2^{n_i}u_i-1,\ d_j = 2^{m_j}v_j-1, i \leq 13, j \leq 15,$
for some odd integers $u,v, u_i, v_j$ and for some $n,m, n_i, m_j \in \N$.

Recall also that $A$ is indecomposable and it does not split. So, $A_1 \not= 1$. Corollary \ref{specialcases}, obtained from Lemma \ref{corol1A=sigmA} and Corollary \ref{finalconditions},  gives an upper bound of each integer $n,m, n_1,\ldots$ appearing in $A$.

\begin{corollary} \label{specialcases}~\\
i) For $i\geq 6$ and $j \geq 9$, one has: $c_i =0$ and $d_j=0$.\\
ii) The integers $u,v,n,m,n_j, m_j,u_j,v_j$ satisfy:\\
$\mbox{$u,v \in \{1,3,5,7,9,13,15\}, u_1 \in \{1,3,5,7,15\}, 1 \leq u_2, u_3, v_1 \leq 3, u_4=u_5= 1$,}$\\
$n,m,n_1 \leq 4, \  n_2, n_3, m_1 \leq 3, \ n_4, n_5 \leq 5 \text{\ and for } j \geq 2, v_j = 1,
m_j \leq 1.$
\end{corollary}

We (quickly) get our theorem in three steps (using {\tt{Maple}}).
First, we dress a list of all $[n,u,m,v,n_1,u_1,n_2,u_2]$ such that $a \geq 1, a \leq b$ and $c_2 =\gamma_2$ (see Lemmas \ref{exponentsofsigmA} and \ref{corol1A=sigmA}). We obtain $10944$ such $8$-tuples. In the second step, all the conditions: $d_j = \delta_j$ give $4484$ $18$-tuples of the form $[n,u,m,v,n_1,u_1,n_2,u_2, d_1, \ldots, d_8,m_1,v_1]$. In the third step, we apply the conditions: $a=\alpha$ and $b=\beta$. We get $44$ polynomials. Among  them, we find the $A$'s such that $a\leq b$ and $\sigma(A) + A$ equals $0$.

\begin{remark}
\emph{Inspired by the proof of Theorem \ref{mainresult1},
if we replace in ${\cal{F}}$, ${\cal{F}}_1$ by the set of all Mersenne primes, then we also get: $T_1, \ldots, T_{11}$.}
\end{remark}

\end{document}